\documentclass
{amsart}
\usepackage{amsfonts,amssymb,amsmath,amsthm}
\usepackage{url}
\usepackage{enumerate}

\newtheorem{thm}{Theorem}[section]
\newtheorem{lem}[thm]{Lemma}
\newtheorem{prop}[thm]{Proposition}

\theoremstyle{definition}
\newtheorem{defn}[thm]{Definition}
\newtheorem{ex}[thm]{Example}

\numberwithin{equation}{section}

\author[J.-Y. Lee, D.-I. Lee]{Jeong-Yup Lee and Dong-il Lee$^*$}
\address{Department of Mathematics Education\\
Catholic Kwandong University\\
Gangwondo 25601, Korea}
\email{jylee@cku.ac.kr}
\address{Department of Mathematics\\
Seoul Women's University\\
Seoul 01797, Korea}
\email{dilee@swu.ac.kr}
\thanks{* 
Corresponding author}

\keywords{
Coxeter group, \GS basis
}
\subjclass[2010]{Primary 20F55, Secondary 05E15, 68W30
}

\newcommand{\F}{\mathbb{F}}
\newcommand{\C}{\mathbb{C}}

\newcommand{\G}{Gr\"{o}bner }
\newcommand{\GS}{Gr\"{o}bner-Shirshov }

\begin{document}

\title[
Coxeter Group $H_4$]{
Gr\"{o}bner-Shirshov bases for
Non-crystallographic
Coxeter Groups 
}

\begin{abstract}
For the noncrystallographic
Coxeter groups of type $H$,
we construct their \GS bases and the corresponding standard monomials.
\end{abstract}
\maketitle

\section{Introduction}

The object in this paper is the symmetry group of
a four-dimensional regular polytope, the 120-cell or its dual, the 600-cell.
We note that this is among a general class of groups, so called, the Coxeter groups.
They appear naturally in geometry and algebra.
In 1935, the finite Coxeter groups were classified by Coxeter in terms of Coxeter-Dynkin diagrams \cite{Coxeter35}.
Contrary to the crystallographic Coxeter groups, 
we remark that the root systems of type $H_k$ $(k=2,3,4)$ and the affine extensions of the Coxeter groups of type $H$
are related to the structure of quasicrystals with five-fold symmetry \cite{ChenMoodyPatera, PateraTwarock}.
Quasicrystals are metalic alloys whose diffraction pattern consists of pure point peaks.
Quasicrystal with five-fold rotational symmetry of diffraction pattern was discovered around mid 1980's by Dan Shechtman.

Our approach to understanding the structure of Coxeter groups is the noncommutative \G basis theory,
which is called the {\em \GS basis theory}.
The theory of \GS bases provides a powerful tool for understanding
the structure of (non)associative algebras and their
representations, especially in computational aspects.
With the ever-growing power of computers,
it is now viewed as a universal engine behind algebraic or symbolic computation.

The effective notion stems from Shirshov's Composition Lemma and his algorithm \cite{Sh, Shirshov_Selectd} for Lie algebras and
independently from Buchberger's algorithm \cite{B65} of computing \G bases for commutative algebras.
In \cite{Bo}, Bokut applied Shirshov's method to associative algebras, and
Bergman mentioned the diamond lemma for ring theory \cite{Be}.

The \GS bases for finite dimensional simple Lie algebras were completely determined by Bokut and Klein \cite{BK96, BK98, BK99}.
The cases for affine Kac-Moody algebras were calculated in \cite{Poro}. 
For basic classical Lie superalgebras of types $A,B,C,D$ and their universal enveloping algebras, Bokut 
et al. developed the corresponding theory
and gave an explicit construction of \GS bases \cite{BKLM}. 
Moreover, exceptional Lie superalgebras were treated in \cite{LeeDI4}.
%

For crystallographic Coxeter groups of classical and exceptional types,
their \GS bases were constructed in \cite{BShiao, 
DenisLee, LeeDI5, LeeDI6, Svechkarenko}.
In this paper, we deal with 
the noncrystallographic Coxeter group of type $H_4$.
By completing the relations coming from a presentation of the Coxeter group,
we find a \GS basis to obtain the set of standard monomials.
Thus the operation table between 11400 standard monomials follows naturally.

\section{Preliminaries}

In this section, we recall a basic theory of {\em \GS bases} for associative algebras to make the paper self-contained.
\\

Let $X$ be a set and let $\langle X\rangle$ be the free monoid of associative
words on $X$. We denote the empty word by $1$ and the {\em length}
(or {\em degree}) of a word $u$ by $l(u)$. A well-ordering $<$ on
$\langle X\rangle$ is called a {\em monomial order} if $x<y$ implies $axb<ayb$
for all $a,b\in \langle X\rangle$.

Fix a monomial order $<$ on $\langle X\rangle$ and let $\mathbb{F}\langle X\rangle$ be the
free associative algebra generated by $X$ over a field $\mathbb{F}$.
Given a nonzero element $p \in \mathbb{F}\langle X\rangle$, we denote by
$\overline{p}$ the maximal monomial (called the {\em leading
monomial}) appearing in $p$ under the ordering $<$. Thus $p = \alpha
\overline{p} + \sum \beta _i w_i $ with $\alpha , \beta _i \in
\mathbb{F}$, $ w_i \in \langle X\rangle$, $\alpha \neq 0$ and $w_i <
\overline{p}$. If $\alpha =1$, $p$ is said to be {\em monic}.

Let $S$ be a subset of monic elements in
$\mathbb{F}\langle X\rangle$, and let $I$ be the two-sided ideal of $\mathbb{F}\langle X\rangle$
generated by $S$. Then we say
that the algebra $A= \mathbb{F}\langle X\rangle /I$ is {\em defined by $S$}

\begin{defn}
Given a subset  $S$  of monic elements in
$\mathbb{F}\langle X\rangle$, a monomial $u \in \langle X\rangle$ is said to be {\em $S$-standard} (or {\em $S$-reduced})
if $u \neq a\overline{s}b$
 for any $s \in S$ and $a, b \in \langle X\rangle$. Otherwise, the monomial $u$ is said to be {\em $S$-reducible}.
\end{defn}

\begin{lem}[\cite{Be, Bo}]\label{division}
Every $p \in \mathbb{F}\langle X\rangle$ can be expressed as
\begin{equation} \label{equ-1}
p = \sum \alpha_i a_is_ib_i + \sum \beta_j
u_j,
\end{equation}
where $\alpha_i, \beta_j \in \mathbb{F}$, $a_i, b_i, u_j \in \langle X\rangle$, $s_i \in S$, $a_i \overline{s_i} b_i \leq
\overline{p}$, $u_j \leq
\overline{p}$ and $u_j$ are $S$-standard.
\end{lem}

\noindent {\it Remark}.
The term $\sum \beta_j u_j$ in the expression (\ref{equ-1}) is
called a {\em normal form} (or a {\em remainder}) of $p$ with
respect to the subset $S$ (and with respect to the monomial order
$<$). In general, a normal form is not unique.
\\

As an immediate corollary of Lemma \ref{division}, we obtain:

\begin{prop} 
The set of $S$-standard monomials spans the algebra
$A=\mathbb{F}\langle X\rangle/I$ defined by the subset $S$, as a vector space over $\mathbb{F}$.
\end{prop}

Let $p$ and $q$ be monic elements in $\F\langle X\rangle$ with leading
monomials $\overline{p}$ and $\overline{q}$. We define the {\em
composition} of $p$ and $q$ as follows.

\begin{defn}
(a) If there exist $a$ and $b$ in $\langle X\rangle$ such that
$\overline{p}a = b\overline{q} = w$ with $l(\overline{p}) > l(b)$,
then the {\em composition of intersection} is defined to be $(p,q)_w = pa -bq$.

(b) If there exist $a$ and $b$ in $\langle X\rangle$ such that $a \neq 1$,
$a\overline{p}b=\overline{q}=w$, then the {\em composition of
inclusion} is defined to be $(p,q)_{a,b} = apb - q$.
\end{defn}


Let $p, q \in \F\langle X\rangle$ and $w \in \langle X\rangle$. We define the {\em
congruence relation} on $\F\langle X\rangle$ as follows: $p \equiv q
\mod (S; w)$ if and only if $p -q = \sum \alpha_i a_i s_i b_i$, where $\alpha_i \in \mathbb{F}$,
$a_i, b_i \in \langle X\rangle$, $s_i \in S$, $a_i
\overline{s_i} b_i < w$.

\begin{defn}
A subset $S$ of monic elements in $\mathbb{F}\langle X\rangle$
is said to be {\em closed under composition} if
\begin{enumerate}
\item[] $(p,q)_w \equiv 0 \mod (S;w)$ and $(p,q)_{a,b} \equiv
0 \mod (S;w)$ for all $p,q \in S$, $a,b \in \langle X\rangle$ whenever the
compositions $(p,q)_w$ and $(p,q)_{a,b}$ are defined.
\end{enumerate}
\end{defn}



The following theorem is a main tool for our results in the next section.

\begin{thm}[\cite{Be, Bo}] \label{cor-1}
Let $S$ be a subset of
monic elements in $\F\langle X\rangle$. Then the following conditions
are equivalent\,{\rm :}
\begin{enumerate}
\item [{\rm (a)}] $S$ is closed under composition.
\item [{\rm (b)}] For each $p \in \mathbb{F}\langle X\rangle$, a normal form of $p$ with respect to $S$ is unique.
\item [{\rm (c)}]
The set of $S$-standard monomials forms an $\F$-linear basis of the
algebra $A=\mathbb{F}\langle X\rangle/I$ defined by $S$.
\end{enumerate}
\end{thm}

\begin{defn}
A subset $S$ of monic elements in $\mathbb{F}\langle X\rangle$ is a
{\em \GS basis} if $S$ satisfies one of the equivalent
conditions in Theorem \ref{cor-1}. In this case, we say that
$S$ is a {\em \GS basis} for the algebra $A$ defined by $S$.
\end{defn}


\section{Coxeter groups of types $H_2$ and $H_3$}

The Coxeter group of type $H_2=I_2(5)$ is just the dihedral group of a regular pentagon. That is the group generated by two reflections $s_1, s_2$ with relation $$(s_1s_2)^5=1.$$
If we let $G$ be the Coxeter group of type $H_2$
and take the complex field $\C$ as our base field, then
the group algebra $\C[G]$ is generated by $s_1,s_2$ with the defining relations $$
s_1^2=1=s_2^2,\quad s_2s_1s_2s_1s_2=s_1s_2s_1s_2s_1.
$$
Set our monomial order to be the degree-lexicographic order with $s_2>s_1$. Denote by $R$ the set of quadratic relations and the braid relation, that is, $$
R=\{s_1^2-1,\ s_2^2-1,\ s_2s_1s_2s_1s_2-s_1s_2s_1s_2s_1\}.$$
Then
the group algebra $\C[G]$ is isomorphic to $\C\langle s_1,s_2\rangle /I$,
where $I$ is the two-sided ideal generated by $R$ in the free associative algebra $\C\langle s_1,s_2\rangle$.
From now on, we identify $\C[G]$ with $\C\langle s_1,s_2\rangle /I$.
We say that the algebra $\C[G]$ is defined by $R$.

Since the number of $R$-standard monomials is 10, the dimension of $\C[G]$, we know that $R$ is already a \GS basis for $\C[G]$.

\begin{prop}
The $R$-standard monomials form the $\C$-basis elements for the type $H_2$. We enumerate as follows\,{\rm :}
\begin{equation}\label{std_monomials_H2}
1,\	s_1,\	s_2,\ s_1s_2,\	s_2s_1,\	s_1s_2s_1,\	s_2s_1s_2,\	s_1s_2s_1s_2,\	s_2s_1s_2s_1,\	s_1s_2s_1s_2s_1.\end{equation}
\end{prop}
\


Now we consider the Coxeter group $G$ of type $H_3$, which is 
called the full icosahedral group.
It is the group generated by three reflections $s_1, s_2, s_3$ with braid relations
$$(s_1s_2)^5=(s_2s_3)^3=1=(s_1s_3)^2.$$
Let $G$ be the Coxeter group of type $H_3$. Then
the group algebra $\C[G]$ is generated by $s_1, s_2, s_3$ with defining relations $$
s_1^2=1=s_2^2=s_3^2,\quad s_2s_1s_2s_1s_2=s_1s_2s_1s_2s_1,\ s_3s_2s_3=s_2s_3s_2,\ s_3s_1=s_1s_3.
$$
We take the same monomial order as above with $s_3> s_2>s_1$, and set \begin{equation}\label{defining_relations_H3}
R=\{s_1^2-1,\ s_2^2-1,\ s_3^2-1,\ s_2s_1s_2s_1s_2-s_1s_2s_1s_2s_1,\ s_3s_2s_3-s_2s_3s_2,\ s_3s_1-s_1s_3\}.\end{equation}
Then 
$\C[G]$ is identified with $\C\langle s_1,s_2,s_3\rangle /I$,
where $I$ is the two-sided ideal generated by $R$ in the free associative algebra $\C\langle s_1,s_2,s_3\rangle$.

\begin{lem} The following relations hold in $\C[G]$\,{\rm :} \label{Lem_H3}

{\rm (a)} $(s_3s_2s_1)s_3-s_2(s_3s_2s_1)=0$,

{\rm (b)} $(s_3s_2s_1)s_3s_2-s_2(s_3s_2s_1s_2)=0$,

{\rm (c)} $(s_3s_2s_1)^2-s_2(s_3s_2s_1s_2s_1)=0$,

{\rm (d)} $(s_3s_2s_1s_2)s_3s_2-s_2(s_3s_2s_1s_2)s_3=0$,

{\rm (e)} $(s_3s_2s_1s_2)s_3s_2s_1-s_2(s_3s_2s_1s_2s_1)s_3=0$,

{\rm (f)} $(s_3s_2s_1s_2)^2-s_2(s_3s_2s_1s_2s_1)s_3s_2=0$,

{\rm (g)} $(s_3s_2s_1s_2s_1)^2-s_2(s_3s_2s_1s_2s_1)s_3s_2s_1s_2=0$.
\end{lem}

\begin{proof}
(a) We calculate that $$
(s_3s_2s_1)s_3=(s_3s_2s_3)s_1=(s_2s_3s_2)s_1$$
by the braid relations $s_1s_3=s_3s_1$ and $s_3s_2s_3=s_2s_3s_2$.

(b)$\sim$(c) 
Just multiplying the previous relation from the right by $s_2$, and by $s_1$, we have the relations.

(d)$\sim$(f) We apply $s_2s_3s_2=s_3s_2s_3$ and $s_1s_3=s_3s_1$ to get our relations.

(g) We obtain that
\begin{eqnarray*}s_3s_2s_1s_2(s_1s_3)s_2s_1s_2s_1&=&s_3s_2s_1s_2s_3(s_1s_2s_1s_2s_1)=s_3s_2s_1(s_2s_3s_2)s_1s_2s_1s_2\\
&=&s_3s_2s_1s_3s_2s_3s_1s_2s_1s_2=s_2(s_3s_2s_1s_2s_1)s_3s_2s_1s_2,\end{eqnarray*} using the braid relations
$s_1s_3=s_3s_1$, $s_1s_2s_1s_2s_1=s_2s_1s_2s_1s_2$ and $s_2s_3s_2=s_3s_2s_3$.
\end{proof}

In the following theorem, we have a monomial basis of $\C[G]$.
Note that this is different from the result in \cite{LeeDI6} since we have the reverse index of generators.

\begin{thm}
We denote by $\widehat{R}$ the set of defining relations $R$
combined with the relations in Lemma \ref{Lem_H3}. Then $\widehat{R}$
forms a \GS basis for $\C[G]$. The corresponding $\widehat{R}$-standard monomials 
are of the form $$M_2M_3,$$ where
$M_2$ is among the standard monomials (\ref{std_monomials_H2}) and
$M_3$ is one of the following\,{\rm :} \begin{eqnarray}\label{std_monomials_H3}
&1,\ s_3,\ s_3s_2,\ s_3s_2s_1,\ s_3s_2s_1s_2,\ s_3s_2s_1s_2s_1,\ (s_3s_2s_1s_2)s_3,\ (s_3s_2s_1s_2s_1)s_3,\\
&(s_3s_2s_1s_2s_1)(s_3s_2), (s_3s_2s_1s_2s_1)(s_3s_2s_1), (s_3s_2s_1s_2s_1)(s_3s_2s_1s_2), (s_3s_2s_1s_2s_1)(s_3s_2s_1s_2)s_3.\nonumber
\end{eqnarray}
\end{thm}

\begin{proof} We have two ways of proving this theorem.

(I) Counting the number of $\widehat{R}$-standard monomials:

It can be easily checked that the product of each $mM$ and $s_i$ $(i=1,2,3)$ is also one of the monomials $mM$ by the relations in $\widehat{R}$.

We enumerate all $\widehat{R}$-standard monomials in $\C\langle s_1,s_2,s_3\rangle$, which are in the above.
The number of $\widehat{R}$-standard monomials is $$
10\times 12=120,$$
which is exactly equal to the order of $G$ (See \cite[\S2.13]{Humphreys90}), the dimension of $\C[G]$ as a $\C$-vector space.
Hence, Theorem \ref{cor-1} leads us that the set $\widehat{R}$ is a \GS basis for $\C[G]$.

(II) Checking that $\widehat{R}$ is closed under composition:

Note that the set $\widehat{R}$ consists of thirteen polynomials. Following Theorem \ref{cor-1} (a), we show that $\widehat{R}$ is closed under composition.
It is proved by checking that
all possible compositions between any two polynomials reduce trivially.

For example, let $p=(s_3s_2s_1)s_3-s_2(s_3s_2s_1)$ in Lemma \ref{Lem_H3} (a) and $q=s_3s_2s_3-s_2s_3s_2$ in (\ref{defining_relations_H3}),
and consider the composition \begin{center}$(p,q)_w$\quad with $w=s_3s_2s_1s_3s_2s_3$.\end{center}
Then we check that $$
p(s_2s_3)-(s_3s_2s_1)q
=-s_2(s_3s_2s_1)(s_2s_3)+(s_3s_2s_1)s_2s_3s_2=0
$$
by Lemma \ref{Lem_H3} (d).

In this way, we check that all possible compositions reduce to 0. Therefore, the set of relations $\widehat{R}$ is closed under composition.
\end{proof}

\noindent {\it Remark}. (1) We have the multiplication table between the $\widehat{R}$-standard monomials.
For example, we multiply $(s_3s_2s_1s_2s_1)(s_3s_2s_1s_2)s_3$ by $s_1$ from the right, to obtain: $$
(s_3s_2s_1s_2s_1)(s_3s_2s_1s_2)s_3\ s_1=(s_3s_2s_1s_2s_1)^2s_3=s_2(s_3s_2s_1s_2s_1)(s_3s_2s_1s_2)s_3$$
by the commutative relation $s_3s_1=s_1s_3$ and the relation (g) in Lemma \ref{Lem_H3}.

(2) We find that the set of the above $12$ monomials in (\ref{std_monomials_H3}) forms the set of right coset representatives of the subgroup of type $H_2$ in the Coxeter group of type $H_3$.

(3) The monomial $$s_1s_2s_1s_2s_1(s_3s_2s_1s_2s_1)(s_3s_2s_1s_2)s_3$$ is the longest element in $G$,
which is of length $15$, the number of positive roots in a root system of type $H_3$.

(4) The 
full icosahedral group $G$ contains the icosahedral group 
as a normal subgroup of index $2$.
Note that the subgroup of rotational symmetries of a regular icosahedron is isomorphic to the $5$th alternating group $A_5$ of order $60$.
Since the elements of $A_5$ are of even degree, we have an explicit monomial expression for $60$ elements of the icosahedral group.

(5) An algorithm using rewriting rules is sketched for obtaining complete presentations of 
Coxeter groups \cite{Borges-Trenard_Perez-Roses, LeChenadec}.

(6) A noncrystallographic root 
lattice induces a 
quasicrystal structure, using the cut and project method, 
which realizes the structure as a projection of a certain strip by a root map in a higher dimensional lattice
\cite{ChenMoodyPatera}.

\section{Main results for $H_4$}

Let $G$ be the Coxeter group of type $H_4$, 
and take the complex field $\C$ as our base field. Then
the group algebra $\C[G]$ is generated by four elements $s_i\ (i=1,2,3,4)$
with the defining relations: \begin{eqnarray}
{\rm (quadratic\;\; relations)} & s_1^2=s_2^2=s_3^2=s_4^2 =1, \label{quadratic relations}\\
& (s_1s_2)^5=(s_2s_3)^3=(s_3s_4)^3 =1=(s_1s_3)^2=(s_1s_4)^2=(s_2s_4)^2.
\end{eqnarray}
We define our monomial order to be the degree-lexicographic order with $s_4>s_3>s_2>s_1$.
Combining the above relations, we get six braid relations:
\begin{eqnarray}
{\rm (braid\;\; relations)}&s_2s_1s_2s_1s_2=s_1s_2s_1s_2s_1,\  s_3s_2s_3=s_2s_3s_2,\  s_4s_3s_4=s_3s_4s_3  \label{braid relations}\\
&s_3s_1=s_1s_3,\ \ s_4s_1=s_1s_4,\ \ s_4s_2=s_2s_4. \nonumber
\end{eqnarray}
We denote by $R$ the set of relations (\ref{quadratic relations}) and (\ref{braid relations}), that is,
\begin{eqnarray}
R&=&\{\ s_1^2-1,\ \ s_2^2-1,\ \ s_3^2-1,\ \ s_4^2-1,\ \ \nonumber \\
&&\ \ s_2s_1s_2s_1s_2-s_1s_2s_1s_2s_1,\ s_3s_2s_3-s_2s_3s_2,\ s_4s_3s_4-s_3s_4s_3\  \nonumber  \\
&&\ \ s_3s_1-s_1s_3,\ \ s_4s_1-s_1s_4,\ \ s_4s_2-s_2s_4,\ \ \}.\nonumber \label{defining_relations}
\end{eqnarray}
Note that the group algebra $\C[G]$ is isomorphic to $\C\langle s_1,s_2,s_3,s_4\rangle /I$,
where $I$ is the two-sided ideal generated by $R$ in the free associative algebra $\C\langle s_1,s_2,s_3,s_4\rangle$.

\begin{lem}[\cite{BShiao
}] The following relations hold in $\C[G]$\,{\rm :} \label{Lem_H4_432}

{\rm (a)} $(s_4s_3s_2)s_4-s_3(s_4s_3s_2)=0$,

{\rm (b)} $(s_4s_3s_2)s_4s_3-s_3s_2(s_4s_3s_2)=0$.
\end{lem}

A reducible monomial is called minimal if its every proper submonomial is not reducible.
Our strategy to find relations is considering all minimal reducible monomials until we obtain the smallest set of standard monomials.
We check the relations in order of the degree of a monomial, one by one.
While the leading monomial of the above relation (b) is not minimal, we are writing explicitly for 
subsequent computations.

\begin{lem} The following relations hold in $\C[G]$\,{\rm :} \label{Lem_H4_4321}

{\rm (a)} $(s_4s_3s_2s_1)s_4-s_3(s_4s_3s_2s_1)=0$,

{\rm (b)} $(s_4s_3s_2s_1)s_4s_3-s_3s_2(s_4s_3s_2s_1)=0$,

{\rm (c)} $(s_4s_3s_2s_1)s_4s_3s_2-s_3s_2(s_4s_3s_2s_1s_2)=0$,

{\rm (d)} $(s_4s_3s_2s_1)^2-s_3s_2(s_4s_3s_2s_1s_2s_1)=0$.
\end{lem}

\begin{proof} (a) The relation follow from the commutativity of $s_4$ with $s_i$ for $i=1,2$ and the braid relation $s_4s_3s_4=s_3s_4s_3$.

(b)$\sim$(d) Multiplying (a) by the alphabet $s_3$ from the right and using the braid relations, we have the relation (b).
Consecutively, we multiply by an alphabet from the right to obtain (c) and (d).
\end{proof}

\begin{lem} The following relations hold in $\C[G]$\,{\rm :} \label{Lem_H4_43212}

{\rm (a)} $(s_4s_3s_2s_1s_2)s_4-s_3(s_4s_3s_2s_1s_2)=0$,

{\rm (b)} $(s_4s_3s_2s_1s_2)s_4s_3-s_3(s_4s_3s_2s_1s_2s_3)=0$,

{\rm (c)} $(s_4s_3s_2s_1s_2)s_4s_3s_2-s_3s_2(s_4s_3s_2s_1s_2s_3)=0$,

{\rm (d)} $(s_4s_3s_2s_1s_2)s_4s_3s_2s_1-s_3s_2(s_4s_3s_2s_1s_2s_1s_3)=0$,

{\rm (e)} $(s_4s_3s_2s_1s_2)^2-s_3s_2(s_4s_3s_2s_1s_2s_1s_3s_2)=0$.
\end{lem}

\begin{proof} We get the relation (a) by applying the commutativity of $s_4$ with $s_i$ for $i=1,2$ and the braid relation $s_4s_3s_4=s_3s_4s_3$.
Similarly, the other relations are deduced from the braid relations.
\end{proof}

\begin{lem} The following relations hold in $\C[G]$\,{\rm :} \label{Lem_H4_432123}

{\rm (a)} $(s_4s_3s_2s_1s_2s_3)s_2-s_2(s_4s_3s_2s_1s_2s_3)=0$,

{\rm (b)} $(s_4s_3s_2s_1s_2s_3)s_4s_3-s_3(s_4s_3s_2s_1s_2s_3)s_4=0$,

{\rm (c)} $(s_4s_3s_2s_1s_2s_3)s_4s_3s_2-s_3s_2(s_4s_3s_2s_1s_2s_3)s_4=0$,

{\rm (d)} $(s_4s_3s_2s_1s_2s_3)s_4s_3s_2s_1-s_3s_2(s_4s_3s_2s_1s_2s_1s_3)s_4=0$,

{\rm (e)} $(s_4s_3s_2s_1s_2s_3)s_4s_3s_2s_1s_2-s_3s_2(s_4s_3s_2s_1s_2s_1s_3s_2)s_4=0$,

{\rm (f)} $(s_4s_3s_2s_1s_2s_3)^2-s_3s_2(s_4s_3s_2s_1s_2s_1s_3s_2)s_4s_3=0$.
\end{lem}

\begin{proof} (a) We have the relation (a) by the braid relations $s_2s_3s_2=s_3s_2s_3$, $s_1s_3=s_3s_1$ and $s_4s_2=s_2s_4$.

(b) Using the braid relation $s_3s_4s_3=s_4s_3s_4$ and the commutativity of $s_4$ with $s_i$ for $i=1,2$, we calculate that
$$s_4s_3s_2s_1s_2(s_3s_4s_3)=s_4s_3s_2s_1s_2(s_4s_3s_4)=(s_4s_3s_4)s_2s_1s_2s_3s_4=(s_3s_4s_3)s_2s_1s_2s_3s_4.$$

(c)$\sim$(f) We multiply the previous relation by a reflection $s_i$ from the right and use the braid relations, obtaining our relations.
\end{proof}

\begin{lem} The following relations hold in $\C[G]$\,{\rm :} \label{Lem_H4_432121}

{\rm (a)} $(s_4s_3s_2s_1s_2s_1)s_4-s_3(s_4s_3s_2s_1s_2s_1)=0$,

{\rm (b)} $(s_4s_3s_2s_1s_2s_1)s_4s_3-s_3(s_4s_3s_2s_1s_2s_1s_3)=0$,

{\rm (c)} $(s_4s_3s_2s_1s_2s_1)s_4s_3s_2-s_3(s_4s_3s_2s_1s_2s_1s_3s_2)=0$,

{\rm (d)} $(s_4s_3s_2s_1s_2s_1)s_4s_3s_2s_1-s_3(s_4s_3s_2s_1s_2s_1s_3s_2s_1)=0$,

{\rm (e)} $(s_4s_3s_2s_1s_2s_1)s_4s_3s_2s_1s_2-s_3(s_4s_3s_2s_1s_2s_1s_3s_2s_1s_2)=0$,

{\rm (f)} $(s_4s_3s_2s_1s_2s_1)s_4s_3s_2s_1s_2s_3-s_3(s_4s_3s_2s_1s_2s_1s_3s_2s_1s_2s_3)=0$,

{\rm (g)} $(s_4s_3s_2s_1s_2s_1)^2-s_3s_2(s_4s_3s_2s_1s_2s_1s_3s_2s_1s_2)=0$.
\end{lem}

\begin{proof} (a) The commutativity of $s_4$ with $s_i$ for $i=1,2$ and the braid relation $s_4s_3s_4=s_3s_4s_3$ leads us the relation.

(b)$\sim$(f) The previous relation is multiplied from the right by a reflection to produce our relation inductively.

(g) Multiplying the relation (e) by $s_1$ from the right and applying $s_1s_2s_1s_2s_1=s_2s_1s_2s_1s_2$ and $s_3s_2s_3=s_2s_3s_2$, we obtain that
$$\begin{aligned}
&(s_4s_3s_2s_1s_2s_1)s_4s_3s_2s_1s_2s_1=s_3(s_4s_3s_2s_1s_2s_1s_3s_2s_1s_2)s_1=s_3(s_4s_3s_2s_1s_2s_3) s_1s_2s_1s_2s_1\\
&=s_3(s_4s_3s_2s_1s_2s_3) s_2s_1s_2s_1s_2=s_3(s_4s_3s_2s_1)s_3s_2s_3 s_1s_2s_1s_2=s_3s_2(s_4s_3s_2s_1s_2 s_1s_3s_2s_1s_2).
\end{aligned}$$
\end{proof}

\begin{lem} The following relations hold in $\C[G]$\,{\rm :} \label{Lem_H4_4321213}

{\rm (a)} $(s_4s_3s_2s_1s_2s_1s_3)s_4s_3-s_3(s_4s_3s_2s_1s_2s_1s_3)s_4=0$,

{\rm (b)} $(s_4s_3s_2s_1s_2s_1s_3)s_4s_3s_2-s_3(s_4s_3s_2s_1s_2s_1s_3s_2)s_4=0$,

{\rm (c)} $(s_4s_3s_2s_1s_2s_1s_3)s_4s_3s_2s_1-s_3(s_4s_3s_2s_1s_2s_1s_3s_2s_1)s_4=0$,

{\rm (d)} $(s_4s_3s_2s_1s_2s_1s_3)s_4s_3s_2s_1s_2-s_3(s_4s_3s_2s_1s_2s_1s_3s_2s_1s_2)s_4=0$,

{\rm (e)} $(s_4s_3s_2s_1s_2s_1s_3)s_4s_3s_2s_1s_2s_3-s_3(s_4s_3s_2s_1s_2s_1s_3s_2s_1s_2)s_4s_3=0$,

{\rm (f)} $(s_4s_3s_2s_1s_2s_1s_3)s_4s_3s_2s_1s_2s_1-s_3s_2(s_4s_3s_2s_1s_2s_1s_3s_2s_1s_2)s_4=0$,

{\rm (g)} $(s_4s_3s_2s_1s_2s_1s_3)^2-s_3s_2(s_4s_3s_2s_1s_2s_1s_3s_2s_1s_2)s_4s_3=0$.
\end{lem}

\begin{proof} (a) We use the braid relation $s_3s_4s_3=s_4s_3s_4$ and the commutativity of $s_4$ with $s_i$ for $i=1,2$
to have the relation (a).

(b)$\sim$(e) Multiplying the previous relation by a reflection from the right, we get our relation inductively.

(f) We multiply the relation (d) by $s_1$ from the right. Then we apply $s_4s_1=s_1s_4$ and the relation (g) in Lemma \ref{Lem_H3}
to obtain our relation from $s_4s_2=s_2s_4$.

(g) The relation (f) is just multiplied from the right by $s_3$ to lead us our relation.
\end{proof}

\begin{lem} The following relations hold in $\C[G]$\,{\rm :} \label{Lem_H4_43212132}

{\rm (a)} $(s_4s_3s_2s_1s_2s_1s_3s_2)s_4s_3s_2-s_3(s_4s_3s_2s_1s_2s_1s_3s_2)s_4s_3=0$,

{\rm (b)} $(s_4s_3s_2s_1s_2s_1s_3s_2)s_4s_3s_2s_1-s_3(s_4s_3s_2s_1s_2s_1s_3s_2s_1)s_4s_3=0$,

{\rm (c)} $(s_4s_3s_2s_1s_2s_1s_3s_2)s_4s_3s_2s_1s_2-s_3(s_4s_3s_2s_1s_2s_1s_3s_2s_1)s_4s_3s_2=0$,

{\rm (d)} $(s_4s_3s_2s_1s_2s_1s_3s_2)s_4s_3s_2s_1s_2s_3-s_3(s_4s_3s_2s_1s_2s_1s_3s_2s_1s_2)s_4s_3s_2=0$,

{\rm (e)} $(s_4s_3s_2s_1s_2s_1s_3s_2)s_4s_3s_2s_1s_2s_1-s_3(s_4s_3s_2s_1s_2s_1s_3s_2s_1)s_4s_3s_2s_1=0$,

{\rm (f)} $(s_4s_3s_2s_1s_2s_1s_3s_2)s_4s_3s_2s_1s_2s_1s_3-s_3(s_4s_3s_2s_1s_2s_1s_3s_2s_1s_2)s_4s_3s_2s_1=0$,

{\rm (g)} $(s_4s_3s_2s_1s_2s_1s_3s_2)^2-s_3(s_4s_3s_2s_1s_2s_1s_3s_2s_1s_2)s_4s_3s_2s_1s_2=0$.
\end{lem}

\begin{proof} (a) Using the commutativity $s_2s_4=s_4s_2$, the braid relation $s_2s_3s_2=s_3s_2s_3$ and
the relation (b) in Lemma \ref{Lem_H4_4321213}, we have: $$\begin{aligned}
&(s_4s_3s_2s_1s_2s_1s_3s_2)s_4s_3s_2=(s_4s_3s_2s_1s_2s_1s_3)s_4s_2s_3s_2\\
&=(s_4s_3s_2s_1s_2s_1s_3)s_4s_3s_2s_3
=s_3(s_4s_3s_2s_1s_2s_1s_3s_2)s_4s_3.\end{aligned}$$

(b) The relation (a) is multiplied by a reflection $s_1$ from the right to produce the relation (b)
after the commutativity of $s_1$ with $s_4s_3$ is applied.

(c) This is just from the multiplication of (b) by $s_2$ from the right.

(d) Multiplying (c) by $s_3$ from the right and using $s_3s_2s_3=s_2s_3s_2$ and $s_4s_2=s_2s_4$, we get our relation.

(e) This relation follows from multiplying (c) by $s_1$ from the right.

(f) We multiply the previous relation by $s_3$ from the right and use the commutativity and the braid relation,
giving us the relation.

(g) This is from the multiplication of (f) by $s_2$.
\end{proof}

\begin{lem} The following relations hold in $\C[G]$\,{\rm :} \label{Lem_H4_432121321}

{\rm (a)} $(s_4s_3s_2s_1s_2s_1s_3s_2s_1)s_4s_3s_2s_1s_2s_1-s_3(s_4s_3s_2s_1s_2s_1s_3s_2s_1)s_4s_3s_2s_1s_2=0$,

{\rm (b)} $(s_4s_3s_2s_1s_2s_1s_3s_2s_1)s_4s_3s_2s_1s_2s_1s_3-s_3(s_4s_3s_2s_1s_2s_1s_3s_2s_1)s_4s_3s_2s_1s_2s_3=0$,

{\rm (c)} $(s_4s_3s_2s_1s_2s_1s_3s_2s_1)s_4s_3s_2s_1s_2s_1s_3s_2-s_3(s_4s_3s_2s_1s_2s_1s_3s_2s_1s_2)s_4s_3s_2s_1s_2s_3=0$,

{\rm (d)} $(s_4s_3s_2s_1s_2s_1s_3s_2s_1)^2-s_3(s_4s_3s_2s_1s_2s_1s_3s_2s_1s_2)s_4s_3s_2s_1s_2s_1s_3=0$.
\end{lem}

\begin{proof} (a) We apply the commutativity of $s_1$ with $s_i (i=3,4)$, the braid relation $s_1s_2s_1s_2s_1=s_2s_1s_2s_1s_2$, and
the relation (e) in Lemma \ref{Lem_H4_43212132} to obtain that
$$\begin{aligned}
&(s_4s_3s_2s_1s_2s_1s_3s_2s_1)s_4s_3s_2s_1s_2s_1=(s_4s_3s_2s_1s_2s_1s_3s_2)s_4s_3(s_1s_2s_1s_2s_1)\\
&=(s_4s_3s_2s_1s_2s_1s_3s_2)s_4s_3s_2s_1s_2s_1 s_2
=s_3(s_4s_3s_2s_1s_2s_1s_3s_2s_1)s_4s_3s_2s_1 s_2.\end{aligned}$$

(b) This is from the multiplication of (a) by $s_3$ from the right.

(c) Multiplying (b) by $s_2$ from the right and using the relations $s_2s_3s_2=s_3s_2s_3$, $s_1s_3=s_3s_1$ and $s_4s_2=s_2s_4$,
we get this relation.

(d) The previous relation is multiplied by $s_1$ and the commutativity $31=13$ is applied to produce our relation.
\end{proof}

\begin{lem} The following relation holds in $\C[G]$\,{\rm :} \label{Lem_H4_4321213212}
$$(s_4s_3s_2s_1s_2s_1s_3s_2s_1s_2)^2-s_3(s_4s_3s_2s_1s_2s_1s_3s_2s_1s_2)s_4s_3s_2s_1s_2s_1s_3s_2s_1=0.$$
\end{lem}

\begin{proof} We calculate that $$\begin{aligned}
&(s_4s_3s_2s_1s_2s_1s_3s_2s_1s_2)s_4s_3s_2s_1s_2s_1s_3s_2s_1s_2
=(s_4s_3s_2s_1s_2s_1s_3s_2s_1)s_4 s_2s_3s_2 s_1s_2s_1s_3s_2s_1s_2\\
&=(s_4s_3s_2s_1s_2s_1s_3s_2s_1)s_4s_3s_2s_3 s_1s_2s_1s_3s_2s_1s_2
=(s_4s_3s_2s_1s_2s_1s_3s_2s_1)s_4s_3s_2s_1 (s_3 s_2s_1)^2s_2\\
&=(s_4s_3s_2s_1s_2s_1s_3s_2s_1)s_4s_3s_2s_1 s_2s_3s_2s_1s_2s_1 s_2
=(s_4s_3s_2s_1s_2s_1s_3s_2s_1)s_4s_3s_2s_1 s_2s_3 s_1s_2s_1s_2s_1\\
&=(s_4s_3s_2s_1s_2s_1s_3s_2s_1)^2 s_2s_1
=s_3(s_4s_3s_2s_1s_2s_1s_3s_2s_1s_2)s_4s_3s_2s_1s_2s_1s_3 s_2s_1
\end{aligned}$$
by applying successively the relations $s_2s_4=s_4s_2$, $s_2s_3s_2=s_3s_2s_3$, $s_3s_1=s_1s_3$, Lemma \ref{Lem_H3} (c),
$s_2s_1s_2s_1s_2=s_1s_2s_1s_2s_1$, $s_3s_1=s_1s_3$, and Lemma \ref{Lem_H4_432121321} (d).
\end{proof}

\begin{lem} The following relations hold in $\C[G]$\,{\rm :} \label{Lem_H4_4321213212_4}

$\begin{aligned}{\rm (a)}\ &(s_4s_3s_2s_1s_2s_1s_3s_2s_1s_2)(s_4s_3s_2s_1s_2s_3)s_4s_3\\
&-(s_4s_3s_2s_1s_2s_1s_3s_2s_1s_2s_3)(s_4s_3s_2s_1s_2s_3)s_4=0,\end{aligned}$

$\begin{aligned}{\rm (b)}\  &(s_4s_3s_2s_1s_2s_1s_3s_2s_1s_2)(s_4s_3s_2s_1s_2s_1s_3)s_4s_3\\
&-(s_4s_3s_2s_1s_2s_1s_3s_2s_1s_2s_3)(s_4s_3s_2s_1s_2s_1s_3)s_4=0,\end{aligned}$

$\begin{aligned}{\rm (c)}\  &(s_4s_3s_2s_1s_2s_1s_3s_2s_1s_2)(s_4s_3s_2s_1s_2s_1s_3s_2)s_4s_3s_2\\
&-(s_4s_3s_2s_1s_2s_1s_3s_2s_1s_2s_3)(s_4s_3s_2s_1s_2s_1s_3s_2)s_4s_3=0,\end{aligned}$

$\begin{aligned}{\rm (d)}\  &(s_4s_3s_2s_1s_2s_1s_3s_2s_1s_2)(s_4s_3s_2s_1s_2s_1s_3s_2s_1)s_4s_3s_2s_1s_2s_1\\
&-(s_4s_3s_2s_1s_2s_1s_3s_2s_1s_2s_3)(s_4s_3s_2s_1s_2s_1s_3s_2s_1)s_4s_3s_2s_1s_2=0,\end{aligned}$

$\begin{aligned}{\rm (e)}\  &(s_4s_3s_2s_1s_2s_1s_3s_2s_1s_2)(s_4s_3s_2s_1s_2s_1s_3s_2s_1)(s_4s_3s_2s_1s_2s_3)s_4s_3\\
&-s_2(s_4s_3s_2s_1s_2s_1s_3s_2s_1s_2)(s_4s_3s_2s_1s_2s_1s_3s_2s_1)(s_4s_3s_2s_1s_2s_3)s_4=0.\end{aligned}$
\end{lem}

\begin{proof} (a)$\sim$(d) Each relation follows from Lemma \ref{Lem_H4_432123} (b),
Lemma \ref{Lem_H4_4321213} (a), Lemma \ref{Lem_H4_43212132} (a), and
Lemma \ref{Lem_H4_432121321} (a), respectively.

(e) We compute that
$$\begin{aligned}
&(s_4s_3s_2s_1s_2s_1s_3s_2s_1s_2)(s_4s_3s_2s_1s_2s_1s_3s_2s_1)(s_4s_3s_2s_1s_2s_3)s_4s_3\\
&=(s_4s_3s_2s_1s_2s_1s_3s_2s_1s_2)(s_4s_3s_2s_1s_2s_1s_3s_2s_1)s_3(s_4s_3s_2s_1s_2s_3)s_4\\
&=(s_4s_3s_2s_1s_2s_1s_3s_2s_1s_2)s_4s_3s_2s_1s_2s_1s_3s_2s_3 s_1 (s_4s_3s_2s_1s_2s_3)s_4\\
&=(s_4s_3s_2s_1s_2s_1s_3s_2s_1s_2)s_4s_3s_2s_1s_2s_1s_2s_3s_2s_1 (s_4s_3s_2s_1s_2s_3)s_4\\
&=(s_4s_3s_2s_1s_2s_1s_3s_2s_1s_2)s_4s_3s_1s_2s_1s_2s_1s_3s_2s_1 (s_4s_3s_2s_1s_2s_3)s_4\\
&=s_4(s_3s_2s_1s_2s_1)^2 (s_4s_3s_2s_1s_2s_1s_3s_2s_1) (s_4s_3s_2s_1s_2s_3)s_4\\
&=s_4s_2(s_3s_2s_1s_2s_1s_3s_2s_1s_2)(s_4s_3s_2s_1s_2s_1s_3s_2s_1) (s_4s_3s_2s_1s_2s_3)s_4\\
&=s_2(s_4s_3s_2s_1s_2s_1s_3s_2s_1s_2)(s_4s_3s_2s_1s_2s_1s_3s_2s_1) (s_4s_3s_2s_1s_2s_3)s_4
\end{aligned}$$
by consecutively applying Lemma \ref{Lem_H4_432123} (b), $s_1s_3=s_3s_1$, $s_3s_2s_3=s_2s_3s_2$, $s_2s_1s_2s_1s_2=s_1s_2s_1s_2s_1$, $s_4s_3s_1=s_1s_4s_3$, Lemma \ref{Lem_H3} (g), and finally $s_4s_2=s_2s_4$.
\end{proof}

\begin{lem} The following relations hold in $\C[G]$\,{\rm :} \label{Lem_H4_43212132123}

{\rm (a)} $(s_4s_3s_2s_1s_2s_1s_3s_2s_1s_2s_3)s_1
-s_2(s_4s_3s_2s_1s_2s_1s_3s_2s_1s_2s_3)=0$,

{\rm (b)} $(s_4s_3s_2s_1s_2s_1s_3s_2s_1s_2s_3)s_2
-s_1(s_4s_3s_2s_1s_2s_1s_3s_2s_1s_2s_3)=0$.
\end{lem}

\begin{proof}
(a) From $s_3s_1=s_1s_3$, Lemma \ref{Lem_H3} (g) and $s_4s_2=s_2s_4$, we get that $$\begin{aligned}
&(s_4s_3s_2s_1s_2s_1s_3s_2s_1s_2s_3)s_1=s_4(s_3s_2s_1s_2s_1)^2 s_3\\
&=s_4s_2(s_3s_2s_1s_2s_1)s_3s_2s_1s_2 s_3=s_2(s_4s_3s_2s_1s_2s_1s_3s_2s_1s_2s_3).\end{aligned}$$

(b) It is computed that $$\begin{aligned}
&s_4s_3s_2s_1s_2s_1(s_3s_2s_1s_2)s_3s_2=s_4s_3s_2s_1s_2s_1 s_2s_3s_2s_1s_2s_3\\
&=s_4s_3s_1s_2s_1s_2s_1 s_3s_2s_1s_2s_3=s_1 (s_4s_3s_2s_1s_2s_1 s_3s_2s_1s_2s_3)\end{aligned}$$
by Lemma \ref{Lem_H3} (d), $s_2s_1s_2s_1 s_2=s_1s_2s_1 s_2s_1$, and the commutativity of $s_1$ with $s_i (i=3,4)$.
\end{proof}

\begin{lem} The following relations hold in $\C[G]$\,{\rm :} \label{Lem_H4_43212132123_4}

$\begin{aligned}{\rm (a)}\ &(s_4s_3s_2s_1s_2s_1s_3s_2s_1s_2s_3)(s_4s_3s_2s_1s_2s_3)s_4s_3\\
&-(s_4s_3s_2s_1s_2s_1s_3s_2s_1s_2)(s_4s_3s_2s_1s_2s_3)s_4=0,\end{aligned}$

$\begin{aligned}{\rm (b)}\ &(s_4s_3s_2s_1s_2s_1s_3s_2s_1s_2s_3)(s_4s_3s_2s_1s_2s_1s_3)s_4s_3\\
&-(s_4s_3s_2s_1s_2s_1s_3s_2s_1s_2)(s_4s_3s_2s_1s_2s_1s_3)s_4=0,\end{aligned}$

$\begin{aligned}{\rm (c)}\ &(s_4s_3s_2s_1s_2s_1s_3s_2s_1s_2s_3)(s_4s_3s_2s_1s_2s_1s_3s_2)s_4s_3s_2\\
&-(s_4s_3s_2s_1s_2s_1s_3s_2s_1s_2)(s_4s_3s_2s_1s_2s_1s_3s_2)s_4s_3=0,\end{aligned}$

$\begin{aligned}{\rm (d)}\ &(s_4s_3s_2s_1s_2s_1s_3s_2s_1s_2s_3)(s_4s_3s_2s_1s_2s_1s_3s_2s_1)s_4s_3s_2s_1s_2s_1\\
&-(s_4s_3s_2s_1s_2s_1s_3s_2s_1s_2)(s_4s_3s_2s_1s_2s_1s_3s_2s_1)s_4s_3s_2s_1s_2=0.\end{aligned}$
\end{lem}

\begin{proof} Each relation is obtained from Lemma \ref{Lem_H4_432123} (b),
Lemma \ref{Lem_H4_4321213} (a),
Lemma \ref{Lem_H4_43212132} (a),
Lemma \ref{Lem_H4_432121321} (a), respectively, and followed by $s_3^2=1$.
\end{proof}

\begin{lem} The following relations hold in $\C[G]$\,{\rm :} \label{Lem_H4_43212132123)4321213212}

$\begin{aligned}{\rm (a)}\ &(s_4s_3s_2s_1s_2s_1s_3s_2s_1s_2s_3)(s_4s_3s_2s_1s_2s_1s_3s_2s_1s_2)s_4s_3s_2s_1s_2s_1s_3s_2s_1\\
&-s_3(s_4s_3s_2s_1s_2s_1s_3s_2s_1s_2s_3)(s_4s_3s_2s_1s_2s_1s_3s_2s_1s_2)s_4s_3s_2s_1s_2s_1s_3s_2=0,\end{aligned}$

$\begin{aligned}{\rm (b)}\ &(s_4s_3s_2s_1s_2s_1s_3s_2s_1s_2s_3)(s_4s_3s_2s_1s_2s_1s_3s_2s_1s_2)(s_4s_3s_2s_1s_2s_3)s_4s_3\\
&-(s_4s_3s_2s_1s_2s_1s_3s_2s_1s_2s_3)^2 (s_4s_3s_2s_1s_2s_3)s_4=0,\end{aligned}$

$\begin{aligned}{\rm (c)}\ &(s_4s_3s_2s_1s_2s_1s_3s_2s_1s_2s_3)(s_4s_3s_2s_1s_2s_1s_3s_2s_1s_2)(s_4s_3s_2s_1s_2s_1s_3)s_4s_3\\
&-(s_4s_3s_2s_1s_2s_1s_3s_2s_1s_2s_3)^2 (s_4s_3s_2s_1s_2s_1s_3)s_4=0,\end{aligned}$

$\begin{aligned}{\rm (d)}\ &(s_4s_3s_2s_1s_2s_1s_3s_2s_1s_2s_3)(s_4s_3s_2s_1s_2s_1s_3s_2s_1s_2)(s_4s_3s_2s_1s_2s_1s_3s_2)s_4s_3s_2\\
&-(s_4s_3s_2s_1s_2s_1s_3s_2s_1s_2s_3)^2 (s_4s_3s_2s_1s_2s_1s_3s_2)s_4s_3=0.\end{aligned}$
\end{lem}

\begin{proof} (a) It is calculated that
$$\begin{aligned}&(s_4s_3s_2s_1s_2s_1s_3s_2s_1s_2s_3)(s_4s_3s_2s_1s_2s_1s_3s_2s_1s_2)s_4s_3s_2s_1s_2s_1s_3s_2s_1\\
&=(s_4s_3s_2s_1s_2s_1s_3s_2s_1s_2) s_4s_3s_2s_1s_2s_1 (s_4s_3s_2s_1s_2)^2 s_1s_3s_2s_1\\
&=(s_4s_3s_2s_1s_2s_1s_3s_2s_1s_2) s_4s_3s_2s_1s_2s_1s_3s_2 s_4s_3s_2s_1s_2s_1(s_3s_2 s_1)^2\\
&=(s_4s_3s_2s_1s_2s_1s_3s_2s_1s_2) s_4s_3s_2s_1s_2s_1s_3s_2 s_4s_3(s_2s_1s_2s_1s_2) s_3s_2s_1s_2s_1\\
&=(s_4s_3s_2s_1s_2s_1s_3s_2s_1s_2) (s_4s_3s_2s_1s_2s_1s_3s_2s_1) s_4(s_3s_2s_1s_2s_1)^2\\
&=(s_4s_3s_2s_1s_2s_1s_3s_2s_1s_2)^2 s_4s_3s_2s_1s_2s_1s_3s_2s_1s_2\\
&=s_3(s_4s_3s_2s_1s_2s_1s_3s_2s_1s_2) (s_4s_3s_2s_1s_2s_1s_3s_2s_1)^2  s_2\\
&=s_3(s_4s_3s_2s_1s_2s_1s_3s_2s_1s_2s_3) (s_4s_3s_2s_1s_2s_1s_3s_2s_1s_2) s_4s_3s_2s_1s_2s_1s_3s_2\end{aligned}$$
from $s_3s_4s_3=s_4s_3s_4$, Lemma \ref{Lem_H4_43212} (e), Lemma \ref{Lem_H3} (c), $s_2s_1s_2s_1s_2=s_1s_2s_1s_2s_1$, Lemma \ref{Lem_H3} (g),
Lemma \ref{Lem_H4_4321213212}, Lemma \ref{Lem_H4_432121321} (d).

(b)$\sim$(d) Each comes directly from Lemma \ref{Lem_H4_4321213212_4} (a), (b), (c), respectively.
\end{proof}

\begin{lem} The following relations hold in $\C[G]$\,{\rm :} \label{Lem_H4_43212132123)^2}



$\begin{aligned}{\rm (a)}\ &(s_4s_3s_2s_1s_2s_1s_3s_2s_1s_2s_3)^2 (s_4s_3s_2s_1s_2s_1s_3s_2s_1)s_4s_3s_2s_1s_2\\
&-s_3(s_4s_3s_2s_1s_2s_1s_3s_2s_1s_2s_3)^2 (s_4s_3s_2s_1s_2s_1s_3s_2s_1)s_4s_3s_2s_1=0,\end{aligned}$

$\begin{aligned}{\rm (b)}\ &(s_4s_3s_2s_1s_2s_1s_3s_2s_1s_2s_3)^2 (s_4s_3s_2s_1s_2s_1s_3s_2s_1s_2)s_4s_3s_2s_1s_2s_3\\
&-s_3(s_4s_3s_2s_1s_2s_1s_3s_2s_1s_2s_3)^2 (s_4s_3s_2s_1s_2s_1s_3s_2s_1s_2)s_4s_3s_2s_1s_2=0,\end{aligned}$

$\begin{aligned}{\rm (c)}\ &(s_4s_3s_2s_1s_2s_1s_3s_2s_1s_2s_3)^2 (s_4s_3s_2s_1s_2s_1s_3s_2s_1s_2)s_4s_3s_2s_1s_2s_1s_3\\
&-s_3(s_4s_3s_2s_1s_2s_1s_3s_2s_1s_2s_3)^2 (s_4s_3s_2s_1s_2s_1s_3s_2s_1s_2)s_4s_3s_2s_1s_2s_1=0.\end{aligned}$
\end{lem}

\begin{proof} 
(a) We obtain that
$$\begin{aligned}&(s_4s_3s_2s_1s_2s_1s_3s_2s_1s_2s_3)s_4s_3s_2s_1s_2s_1s_3s_2s_1s_2s_3 (s_4s_3s_2s_1s_2s_1s_3s_2s_1)s_4s_3s_2s_1s_2\\
&=(s_4s_3s_2s_1s_2s_1s_3s_2s_1s_2s_3)(s_4s_3s_2s_1s_2s_1s_3s_2s_1s_2) s_4s_3s_2s_1s_2s_1 (s_4s_3s_2s_1)^2 s_2\\
&=(s_4s_3s_2s_1s_2s_1s_3s_2s_1s_2s_3)(s_4s_3s_2s_1s_2s_1s_3s_2s_1s_2) s_4s_3s_2s_1s_2s_1s_3s_2 s_4s_3(s_2s_1s_2s_1s_2)\\
&=(s_4s_3s_2s_1s_2s_1s_3s_2s_1s_2s_3)(s_4s_3s_2s_1s_2s_1s_3s_2s_1s_2) (s_4s_3s_2s_1s_2s_1s_3s_2s_1) s_4s_3s_2s_1s_2s_1\\
&=s_3(s_4s_3s_2s_1s_2s_1s_3s_2s_1s_2s_3)(s_4s_3s_2s_1s_2s_1s_3s_2s_1s_2) (s_4s_3s_2s_1s_2s_1s_3s_2) s_4s_3s_2s_1s_2s_1\\
&=s_3(s_4s_3s_2s_1s_2s_1s_3s_2s_1s_2s_3)^2 (s_4s_3s_2s_1s_2s_1s_3s_2s_1) s_4s_3s_2s_1
\end{aligned}$$
by consecutively applying the relations $s_3s_4s_3=s_4s_3s_4$, Lemma \ref{Lem_H4_4321} (d), $s_2s_1s_2s_1s_2=s_1s_2s_1s_2s_1$,
Lemma \ref{Lem_H4_43212132123)4321213212} (a), Lemma \ref{Lem_H4_43212132} (e).

(b) From $s_2s_4=s_4s_2$, $s_2s_3s_2=s_3s_2s_3$, $s_3s_1=s_1s_3$, $s_3s_2s_3=s_2s_3s_2$,
the previous relation (a), $s_1s_3=s_3s_1$, $s_3s_2s_3=s_2s_3s_2$, $s_4s_2=s_2s_4$,
we calculate that $$\begin{aligned}
&(s_4s_3s_2s_1s_2s_1s_3s_2s_1s_2s_3)^2 (s_4s_3s_2s_1s_2s_1s_3s_2s_1s_2)s_4s_3s_2s_1s_2s_3\\
&=(s_4s_3s_2s_1s_2s_1s_3s_2s_1s_2s_3)^2 (s_4s_3s_2s_1s_2s_1s_3s_2s_1)s_4s_2s_3s_2s_1s_2s_3\\
&=(s_4s_3s_2s_1s_2s_1s_3s_2s_1s_2s_3)^2 (s_4s_3s_2s_1s_2s_1s_3s_2s_1)s_4s_3s_2s_3s_1s_2s_3\\
&=(s_4s_3s_2s_1s_2s_1s_3s_2s_1s_2s_3)^2 (s_4s_3s_2s_1s_2s_1s_3s_2s_1)s_4s_3s_2s_1s_3s_2s_3\\
&=(s_4s_3s_2s_1s_2s_1s_3s_2s_1s_2s_3)^2 (s_4s_3s_2s_1s_2s_1s_3s_2s_1)s_4s_3s_2s_1s_2 s_3s_2\\
&=s_3(s_4s_3s_2s_1s_2s_1s_3s_2s_1s_2s_3)^2 (s_4s_3s_2s_1s_2s_1s_3s_2s_1) s_4s_3s_2s_1s_3s_2\\
&=s_3(s_4s_3s_2s_1s_2s_1s_3s_2s_1s_2s_3)^2 (s_4s_3s_2s_1s_2s_1s_3s_2s_1) s_4s_3s_2s_3s_1s_2\\
&=s_3(s_4s_3s_2s_1s_2s_1s_3s_2s_1s_2s_3)^2 (s_4s_3s_2s_1s_2s_1s_3s_2s_1) s_4s_2s_3s_2s_1s_2\\
&=s_3(s_4s_3s_2s_1s_2s_1s_3s_2s_1s_2s_3)^2 (s_4s_3s_2s_1s_2s_1s_3s_2s_1s_2) s_4s_3s_2s_1s_2.
\end{aligned}$$

(c) This is just from multiplying (b) from the right by $s_1$.
\end{proof}

\begin{lem} The following relations hold in $\C[G]$\,{\rm :} \label{Lem_H4_43212132123)^3}


$\begin{aligned}{\rm (a)}\ &(s_4s_3s_2s_1s_2s_1s_3s_2s_1s_2s_3)^3 (s_4s_3s_2s_1s_2s_3)s_4\\
&-s_3(s_4s_3s_2s_1s_2s_1s_3s_2s_1s_2s_3)^3 s_4s_3s_2s_1s_2s_3=0,\end{aligned}$

$\begin{aligned}{\rm (b)}\ &(s_4s_3s_2s_1s_2s_1s_3s_2s_1s_2s_3)^3 (s_4s_3s_2s_1s_2s_1s_3)s_4\\
&-s_3(s_4s_3s_2s_1s_2s_1s_3s_2s_1s_2s_3)^3 s_4s_3s_2s_1s_2s_1s_3=0,\end{aligned}$

$\begin{aligned}{\rm (c)}\ &(s_4s_3s_2s_1s_2s_1s_3s_2s_1s_2s_3)^3 (s_4s_3s_2s_1s_2s_1s_3s_2)s_4\\
&-s_3(s_4s_3s_2s_1s_2s_1s_3s_2s_1s_2s_3)^3 s_4s_3s_2s_1s_2s_1s_3s_2=0,\end{aligned}$

$\begin{aligned}{\rm (d)}\ &(s_4s_3s_2s_1s_2s_1s_3s_2s_1s_2s_3)^3 (s_4s_3s_2s_1s_2s_1s_3s_2s_1)s_4\\
&-s_3(s_4s_3s_2s_1s_2s_1s_3s_2s_1s_2s_3)^3 s_4s_3s_2s_1s_2s_1s_3s_2s_1=0,\end{aligned}$

$\begin{aligned}{\rm (e)}\ &(s_4s_3s_2s_1s_2s_1s_3s_2s_1s_2s_3)^3 (s_4s_3s_2s_1s_2s_1s_3s_2s_1s_2)s_4\\
&-s_3(s_4s_3s_2s_1s_2s_1s_3s_2s_1s_2s_3)^3 s_4s_3s_2s_1s_2s_1s_3s_2s_1s_2=0.\end{aligned}$
\end{lem}

\begin{proof} (a) Using Lemma \ref{Lem_H4_432123} (b), Lemma \ref{Lem_H4_43212132123)^2} (b), Lemma \ref{Lem_H4_43212} (b),
we deduce that
$$\begin{aligned}
&(s_4s_3s_2s_1s_2s_1s_3s_2s_1s_2s_3)^2 (s_4s_3s_2s_1s_2s_1s_3s_2s_1s_2)s_3 (s_4s_3s_2s_1s_2s_3)s_4\\
&=(s_4s_3s_2s_1s_2s_1s_3s_2s_1s_2s_3)^2 (s_4s_3s_2s_1s_2s_1s_3s_2s_1s_2) (s_4s_3s_2s_1s_2s_3)s_4s_3\\
&=s_3(s_4s_3s_2s_1s_2s_1s_3s_2s_1s_2s_3)^2 (s_4s_3s_2s_1s_2s_1s_3s_2s_1s_2)s_4s_3s_2s_1s_2 s_4s_3\\
&=s_3(s_4s_3s_2s_1s_2s_1s_3s_2s_1s_2s_3)^3 s_4s_3s_2s_1s_2s_3.\end{aligned}$$

(b) In the same way as in (a), applying Lemma \ref{Lem_H4_4321213} (a), Lemma \ref{Lem_H4_43212132123)^2} (c), Lemma \ref{Lem_H4_432121} (b)
results in our relation.

(c)$\sim$(e) We multiply the previous relation from the right by a reflection $s_i$ ($i=2,1,2$, successively), showing the relation.
\end{proof}


Finally, we consider the relation of the longest length in our \GS basis computation.

\begin{lem} The following relation holds in $\C[G]$\,{\rm :} \label{Lem_H4_longest}

$$(s_4s_3s_2s_1s_2s_1s_3s_2s_1s_2s_3)^4 s_4s_3-s_3(s_4s_3s_2s_1s_2s_1s_3s_2s_1s_2s_3)^4 s_4=0.$$
\end{lem}

\begin{proof} 
We obtain that
$$\begin{aligned}
&(s_4s_3s_2s_1s_2s_1s_3s_2s_1s_2s_3)^3(s_4s_3s_2s_1s_2s_1s_3s_2s_1s_2)s_3 s_4s_3\\
&=(s_4s_3s_2s_1s_2s_1s_3s_2s_1s_2s_3)^3(s_4s_3s_2s_1s_2s_1s_3s_2s_1s_2)s_4 s_3 s_4\\
&=s_3(s_4s_3s_2s_1s_2s_1s_3s_2s_1s_2s_3)^4 s_4
\end{aligned}$$
by the braid relation $s_3s_4s_3=s_4s_3s_4$ and the relation (e) in Lemma \ref{Lem_H4_43212132123)^3},
which proves our final Lemma.
\end{proof}

Now we have a \GS basis for $\C[G]$ and the corresponding standard monomials.

\begin{thm}\label{Thm_H4}
We denote by $\widehat{R}$ the set of defining relations $R$
combined with the relations in Lemma \ref{Lem_H3} and Lemmas \ref{Lem_H4_432}$\sim$\ref{Lem_H4_longest}. Then $\widehat{R}$
forms a \GS basis for $\C[G]$. The corresponding $\widehat{R}$-standard monomials 
are of the form $$M_2M_3M_4,$$ where
$M_2$ is among the standard monomials (\ref{std_monomials_H2})
, $M_3$ is among the standard monomials (\ref{std_monomials_H3}) 
and $M_4$ is one of the following 120 monomials
\,{\rm :}
\begin{flushleft}
$1$,
$s_4$,
$s_4s_3$,
$s_4s_3s_2$,
$s_4s_3s_2s_1$,
$s_4s_3s_2s_1s_2$,
$s_4s_3s_2s_1s_2s_3$,
$(s_4s_3s_2s_1s_2s_3)s_4$,
$s_4s_3s_2s_1s_2s_1$,
$s_4s_3s_2s_1s_2s_1s_3$,
$(s_4s_3s_2s_1s_2s_1s_3)s_4$,
$s_4s_3s_2s_1s_2s_1s_3s_2$, $(s_4s_3s_2s_1s_2s_1s_3s_2)s_4$, $(s_4s_3s_2s_1s_2s_1s_3s_2)s_4s_3$,
$s_4s_3s_2s_1s_2s_1s_3s_2s_1$,
$(s_4s_3s_2s_1s_2s_1s_3s_2s_1)s_4$,
$(s_4s_3s_2s_1s_2s_1s_3s_2s_1)s_4s_3$,
$(s_4s_3s_2s_1s_2s_1s_3s_2s_1)s_4s_3s_2$,
$(s_4s_3s_2s_1s_2s_1s_3s_2s_1)s_4s_3s_2s_1$,
$(s_4s_3s_2s_1s_2s_1s_3s_2s_1)s_4s_3s_2s_1s_2$,
$(s_4s_3s_2s_1s_2s_1s_3s_2s_1)s_4s_3s_2s_1s_2s_3$
, $\thicksim s_4$,
$s_4s_3s_2s_1s_2s_1s_3s_2s_1s_2$,
$(s_4s_3s_2s_1s_2s_1s_3s_2s_1s_2)s_4$,
$(s_4s_3s_2s_1s_2s_1s_3s_2s_1s_2)s_4s_3$,
$(s_4s_3s_2s_1s_2s_1s_3s_2s_1s_2)s_4s_3s_2$,
$(s_4s_3s_2s_1s_2s_1s_3s_2s_1s_2)s_4s_3s_2s_1$,
$(s_4s_3s_2s_1s_2s_1s_3s_2s_1s_2)s_4s_3s_2s_1s_2$,
$(s_4s_3s_2s_1s_2s_1s_3s_2s_1s_2)s_4s_3s_2s_1s_2s_3$
, $\thicksim s_4$,
$(s_4s_3s_2s_1s_2s_1s_3s_2s_1s_2)s_4s_3s_2s_1s_2s_1$,
$(s_4s_3s_2s_1s_2s_1s_3s_2s_1s_2)s_4s_3s_2s_1s_2s_1s_3$
, $\thicksim s_4$,
$(s_4s_3s_2s_1s_2s_1s_3s_2s_1s_2)s_4s_3s_2s_1s_2s_1s_3s_2$
, $\thicksim s_4$
, $\thicksim s_4s_3$,
$(s_4s_3s_2s_1s_2s_1s_3s_2s_1s_2)s_4s_3s_2s_1s_2s_1s_3s_2s_1$
, $\thicksim s_4$
, $\thicksim s_4s_3$
, $\thicksim s_4s_3s_2$
, $\thicksim s_4s_3s_2s_1$
, $\thicksim s_4s_3s_2s_1s_2$
, $\thicksim s_4s_3s_2s_1s_2s_3$
, $\thicksim (s_4s_3s_2s_1s_2s_3)s_4$,
$s_4s_3s_2s_1s_2s_1s_3s_2s_1s_2s_3$,
$(s_4s_3s_2s_1s_2s_1s_3s_2s_1s_2s_3)s_4$,
$(s_4s_3s_2s_1s_2s_1s_3s_2s_1s_2s_3)s_4s_3$,
$(s_4s_3s_2s_1s_2s_1s_3s_2s_1s_2s_3)s_4s_3s_2$,
$(s_4s_3s_2s_1s_2s_1s_3s_2s_1s_2s_3)s_4s_3s_2s_1$,
$(s_4s_3s_2s_1s_2s_1s_3s_2s_1s_2s_3)s_4s_3s_2s_1s_2$,
$(s_4s_3s_2s_1s_2s_1s_3s_2s_1s_2s_3)s_4s_3s_2s_1s_2s_3$
, $\thicksim s_4$,
$(s_4s_3s_2s_1s_2s_1s_3s_2s_1s_2s_3)s_4s_3s_2s_1s_2s_1$,
$(s_4s_3s_2s_1s_2s_1s_3s_2s_1s_2s_3)s_4s_3s_2s_1s_2s_1s_3$
, $\thicksim s_4$,
$(s_4s_3s_2s_1s_2s_1s_3s_2s_1s_2s_3)s_4s_3s_2s_1s_2s_1s_3s_2$
, $\thicksim s_4$
, $\thicksim s_4s_3$,
$(s_4s_3s_2s_1s_2s_1s_3s_2s_1s_2s_3)s_4s_3s_2s_1s_2s_1s_3s_2s_1$
, $\thicksim s_4$
, $\thicksim s_4s_3$
, $\thicksim s_4s_3s_2$
, $\thicksim s_4s_3s_2s_1$
, $\thicksim s_4s_3s_2s_1s_2$
, $\thicksim s_4s_3s_2s_1s_2s_3$
, $\thicksim (s_4s_3s_2s_1s_2s_3)s_4$,
$(s_4s_3s_2s_1s_2s_1s_3s_2s_1s_2s_3)s_4s_3s_2s_1s_2s_1s_3s_2s_1s_2$
, $\thicksim s_4$
, $\thicksim s_4s_3$
, $\thicksim s_4s_3s_2$
, $\thicksim s_4s_3s_2s_1$
, $\thicksim s_4s_3s_2s_1s_2$
, $\thicksim s_4s_3s_2s_1s_2s_3$
, $\thicksim (s_4s_3s_2s_1s_2s_3)s_4$
, $\thicksim s_4s_3s_2s_1s_2s_1$
, $\thicksim s_4s_3s_2s_1s_2s_1s_3$
, $\thicksim (s_4s_3s_2s_1s_2s_1s_3)s_4$
, $\thicksim s_4s_3s_2s_1s_2s_1s_3s_2$
, $\thicksim (s_4s_3s_2s_1s_2s_1s_3s_2)s_4$
, $\thicksim (s_4s_3s_2s_1s_2s_1s_3s_2)s_4s_3$,
$(s_4s_3s_2s_1s_2s_1s_3s_2s_1s_2s_3)^2$,
$(s_4s_3s_2s_1s_2s_1s_3s_2s_1s_2s_3)^2s_4$,
$(s_4s_3s_2s_1s_2s_1s_3s_2s_1s_2s_3)^2s_4s_3$,
$(s_4s_3s_2s_1s_2s_1s_3s_2s_1s_2s_3)^2s_4s_3s_2$,
$(s_4s_3s_2s_1s_2s_1s_3s_2s_1s_2s_3)^2s_4s_3s_2s_1$,
$(s_4s_3s_2s_1s_2s_1s_3s_2s_1s_2s_3)^2s_4s_3s_2s_1s_2$,
$(s_4s_3s_2s_1s_2s_1s_3s_2s_1s_2s_3)^2s_4s_3s_2s_1s_2s_3$
, $\thicksim s_4$,
$(s_4s_3s_2s_1s_2s_1s_3s_2s_1s_2s_3)^2s_4s_3s_2s_1s_2s_1$,
$(s_4s_3s_2s_1s_2s_1s_3s_2s_1s_2s_3)^2s_4s_3s_2s_1s_2s_1s_3$
, $\thicksim s_4$,
$(s_4s_3s_2s_1s_2s_1s_3s_2s_1s_2s_3)^2s_4s_3s_2s_1s_2s_1s_3s_2$
, $\thicksim s_4$
, $\thicksim s_4s_3$,
$(s_4s_3s_2s_1s_2s_1s_3s_2s_1s_2s_3)^2s_4s_3s_2s_1s_2s_1s_3s_2s_1$
, $\thicksim s_4$
, $\thicksim s_4s_3$
, $\thicksim s_4s_3s_2$
, $\thicksim s_4s_3s_2s_1$,
$(s_4s_3s_2s_1s_2s_1s_3s_2s_1s_2s_3)^2s_4s_3s_2s_1s_2s_1s_3s_2s_1s_2$
, $\thicksim s_4$
, $\thicksim s_4s_3$
, $\thicksim s_4s_3s_2$
, $\thicksim s_4s_3s_2s_1$
, $\thicksim s_4s_3s_2s_1s_2$
, $\thicksim s_4s_3s_2s_1s_2s_1$,
$(s_4s_3s_2s_1s_2s_1s_3s_2s_1s_2s_3)^3$,
$(s_4s_3s_2s_1s_2s_1s_3s_2s_1s_2s_3)^3s_4$,
$(s_4s_3s_2s_1s_2s_1s_3s_2s_1s_2s_3)^3s_4s_3$,
$(s_4s_3s_2s_1s_2s_1s_3s_2s_1s_2s_3)^3s_4s_3s_2$,
$(s_4s_3s_2s_1s_2s_1s_3s_2s_1s_2s_3)^3s_4s_3s_2s_1$,
$(s_4s_3s_2s_1s_2s_1s_3s_2s_1s_2s_3)^3s_4s_3s_2s_1s_2$,
$(s_4s_3s_2s_1s_2s_1s_3s_2s_1s_2s_3)^3s_4s_3s_2s_1s_2s_3$,
$(s_4s_3s_2s_1s_2s_1s_3s_2s_1s_2s_3)^3s_4s_3s_2s_1s_2s_1$,
$(s_4s_3s_2s_1s_2s_1s_3s_2s_1s_2s_3)^3s_4s_3s_2s_1s_2s_1s_3$,
$(s_4s_3s_2s_1s_2s_1s_3s_2s_1s_2s_3)^3s_4s_3s_2s_1s_2s_1s_3s_2$,
$(s_4s_3s_2s_1s_2s_1s_3s_2s_1s_2s_3)^3s_4s_3s_2s_1s_2s_1s_3s_2s_1$,
$(s_4s_3s_2s_1s_2s_1s_3s_2s_1s_2s_3)^3s_4s_3s_2s_1s_2s_1s_3s_2s_1s_2$,
$(s_4s_3s_2s_1s_2s_1s_3s_2s_1s_2s_3)^4$,
$(s_4s_3s_2s_1s_2s_1s_3s_2s_1s_2s_3)^4s_4$,
\end{flushleft}
where, for monomials $m,m_1\ldots,m_n$, the enumeration $m,\thicksim m_1,\ldots,\thicksim m_n$ means $m,mm_1,\ldots,mm_n$.
\end{thm}

\begin{proof}
We can check that the product of each $M_2M_3M_4$ and $s_i$ $(i=1,2,3,4)$ is also one of the monomials $M_2M_3M_4$ by the relations in $\widehat{R}$, and enumerated all $\widehat{R}$-standard monomials in $\C\langle s_1,s_2,s_3\rangle$, as in the above.

The number of $\widehat{R}$-standard monomials is $$
10\times 12\times 120=14400,$$
which is exactly equal to the order of $G$ (See \cite[\S2.13]{Humphreys90}), the dimension of $\C[G]$ as a $\C$-vector space.
Hence, Theorem \ref{cor-1} leads us that the set $\widehat{R}$ is a \GS basis for $\C[G]$.
\end{proof}

\begin{ex} Since we have a \GS basis and the monomial basis for $\C[G]$, we have the multiplication table between the 11400 standard monomials.
For example, if we multiply an element $(s_4s_3s_2s_1s_2s_1s_3s_2s_1s_2s_3)^4 s_4$ from the right by $s_2$ then $$\begin{aligned}
(s_4s_3s_2s_1s_2s_1s_3s_2s_1s_2s_3)^4 s_4 s_2&=(s_4s_3s_2s_1s_2s_1s_3s_2s_1s_2s_3)^3 (s_4s_3s_2s_1s_2s_1s_3s_2s_1s_2s_3)s_2 s_4\\
&=(s_4s_3s_2s_1s_2s_1s_3s_2s_1s_2s_3)^3  s_1(s_4s_3s_2s_1s_2s_1s_3s_2s_1s_2s_3) s_4\\
&=(s_4s_3s_2s_1s_2s_1s_3s_2s_1s_2s_3)^2 s_2(s_4s_3s_2s_1s_2s_1s_3s_2s_1s_2s_3)^2 s_4\\
&=(s_4s_3s_2s_1s_2s_1s_3s_2s_1s_2s_3)s_1 (s_4s_3s_2s_1s_2s_1s_3s_2s_1s_2 s_3)^3  s_4\\
&=s_2(s_4s_3s_2s_1s_2s_1s_3s_2s_1s_2s_3)^4 s_4
\end{aligned}$$
from applying $s_4s_2=s_2s_4$ and the relations in Lemma \ref{Lem_H4_43212132123}.
\end{ex}

\noindent {\it Remark}. (1) We find that the set of the above $120$ standard monomials in Theorem \ref{Thm_H4}
forms the set of right coset representatives of the subgroup of type $H_3$ in the Coxeter group of type $H_4$.

(2) The monomial $$s_1s_2s_1s_2s_1(s_3s_2s_1s_2s_1)(s_3s_2s_1s_2)s_3 (s_4s_3s_2s_1s_2s_1s_3s_2s_1s_2s_3)^4 s_4$$
is the longest element in $G$,
which is of length $60$, the number of positive roots in a root system of type $H_4$.

\vskip 10mm
\proof[Acknowledgements]
The research of the first author was supported by 
NRF Grant \# 2017078374,
and the corresponding author's research was supported by NRF Grant \# 
2018R1D1A1B07044111
.

\vskip 10mm

\end{document}